\documentclass[12pt]{amsart}
\usepackage{amsmath,amsthm,epsfig,amssymb}
\usepackage{amsthm,amsfonts,amssymb,amscd}
\usepackage{tikz}
\usetikzlibrary{positioning}
\newtheorem{thm}[subsection]{Theorem}

\newtheorem{lem}[subsection]{Lemma}

\newtheorem{rem}[subsection]{Remark}
\theoremstyle{definition}
\newtheorem{Def}[subsection]{Definition}

\newtheorem{proposition-definition}[subsection]{Proposition-Definition}

\newcommand{\Oh}{\mathcal{O}}

\begin{document}

\title {Semistability of certain bundles on second symmetric power of a curve}

\author[K. Dan]{Krishanu Dan}
\address{Institute of Mathematical Sciences, C.I.T Campus, Tharamani, Chennai-600113, India}
\email{bubukrish@gmail.com}

\author[S.Pal]{Sarbeswar Pal}

\address{Indian Statistical Institute, 8th Mile, Mysore Road,
Bangalore- 560059, India}

\email{sarbeswar11@gmail.com}

\keywords{vector bundles,  symmetric power, semistability, 
}
\date{}

\begin{abstract}
Let $C$ be a smooth irreducible projective curve and $E$ be a stable bundle of rank $2$ on $C$. Then one can associate a rank $4$ vector bundle
$\mathcal{F}_2(E)$ on $S^2(C)$, the second symmetric power of $C$. Our goal in this article is to study semistability of this bundle. 

\end{abstract}
\maketitle

\section{Introduction}
It has been an interesting and important object to study vector bundles over smooth projective varieties. The moduli space of 
semistable vector bundles with fixed topological invariants is well understood for the case of curves. However 
the question of existence of such bundles is  open  for higher dimensional varieties. In this article we will study the semistability
of certain vector bundles on second symmetric power of a smooth projective curve, which arises naturally.

Let $C$ be smooth irreducible projective curve over the fields $\mathbb{C}$ of complex numbers and $E$ be a rank $r$ vector bundle on $C$.
There is a naturally associated vector bundle $\mathcal{F}_2(E)$ of rank $2r$ on the second symmetric power $S^2(C)$ which is defined 
in Section \ref{s1}. The stability and semi-stability for case $r=1,$ i.e. when $E$ is  a line bundle on $C,$ has been studied and well understood 
(\cite {BN2}, \cite{LMN}). In this article we consider the case when rank $E$ is two.

Fixing a point $x \in C$, the image of $\{x\} \times C$ in $S^2(C)$ defines an ample divisor $H'$ on $S^2(C)$, which we denote by $x+C$.
We prove the following:
\begin{thm}
 Let $E$ be a rank two stable vector bundle of even degree $d \geq 2$ on $C$ such that $\mathcal{F}_2(E)$ is globally generated. 
 Then the bundle $\mathcal{F}_2(E)$ on $S^2(C)$ is $\mu_{H'}$-semistable with respect to the ample class $H'=x+C$.
\end{thm}
 
\begin{thm}
 Assume the genus of $C$ greater than 2. Let $E$ be a rank two $(0, 1)$-stable bundle (defined in Section \ref{s3}) of odd
 degree $d \geq 1$ on $C$ such that $\mathcal{F}_2(E)$ is globally generated.
  Then the bundle $\mathcal{F}_2(E)$ on $S^2(C)$ is $\mu_{H'}$-semistable with respect to the ample class $H'=x+C$.
\end{thm}
\section{Preliminaries}\label{s1}

Let $C$ be a smooth irreducible projective curve over the field of complex numbers $\mathbb{C}$ of genus $g$. On the space $C \times C$, consider the following involution
$C \times C \longrightarrow C \times C, (x, y) \mapsto (y, x)$. The resulting quotient space is denoted by $S^2(C)$, called the second symmetric power of $C$. It is a smooth irreducible projctive 
surface over $\mathbb{C}$.  Note that, $S^2(C)$ is 
naturally identified with the set of all degree $2$ effective divisors of $C$. Set
\[
 \Delta_2:= \{(D, p) \in S^2(C) \times C \rvert D= p + q, \text{for some } q \in C\}.
\]
Then $\Delta_2$ is a divisor in $S^2(C) \times C,$ called the universal divisor of degree $2.$ Let $q_1$ and $q_2$ be the projections from $S^2(C)\times C$ onto the first and second factors 
respectively. Then the restriction of the first projection to $\Delta_2$ induces a  morphism
\[
 q: \Delta_2 \longrightarrow S^2(C),
\]
which is a two sheeted ramified covering.
For any vector bundle $E$ of rank $r$ on $C$ we constract a bundle $\mathcal{F}_2(E):= (q)_*(q_2^*(E)\mid_{\Delta_2})$ of rank $2r$ over $S^2(C).$ From the exact sequence 
\[
 0 \to \mathcal{O}_{S^2(C) \times C}(-\Delta_2) \to \Oh_{S^2(C) \times C} \to \Oh_{\Delta_2} \to 0
\]
on $S^2(C) \times C$ we get the following exact sequence on $S^2(C)$
\[
 0 \to q_{1*}(q_2^*E \otimes \Oh_{S^2(C) \times C}(-\Delta_2)) \to q_{1*}q_2^*E \to \mathcal{F}_2(E).
\]

Define $f: C\times C \to \Delta_2$ by $(x,y) \mapsto (x+y, x).$ Then $f$ is an identification. Let $p_i:C\times C \to C$ be the $i$-th coordinate projection and let $\pi: C\times C \to S^2(C)$ be 
the quotient map. Then it's easy to check that $\pi=q\circ f$ and $\mathcal{F}_2(E) = \pi_*p_2^*E.$

\begin{rem}
 Let $C$ be a smooth irreducible projective curve over $\mathbb{C}$ of genus $g$ and let $M$ be a line bundle on $C$ of degree $d$. Consider the rank two vector bundle $V(M) := \pi_* p_2^*M$ 
 on $S^2(C)$. Using Grothendieck-Riemann-Roch, one can compute the Chern classes of $V(M)$:
 $$
 c_1(V(M)) = (d - g - 1)x + \theta
 $$
 and
 $$
 c_2(V(M)) = {{d - g} \choose 2}x^2 + (d - g)x.\theta + \frac{\theta^2}{2}
 $$
 where $x$ is the image of the cohomology class of $x + C$ in $S^2(C)$, $\theta$ is the cohomology class of the pull back of the theta divisor in Pic$^2(C)$ under the natural map of $S^2(C)$ 
 to Pic$^2(C)$ \cite[Lemma\, 2.5, Chapter\, VIII ]{ACGH}. Note that the cohomology group $H^4(S^2(C), \mathbb{Z})$ is naturally isomorphic to $\mathbb{Z}$, and $x^2 = 1, x.\theta = g, 
 \theta^2 = g(g - 1)$.
 
 To find the Chern character of $\mathcal{F}_2(E),$ for any rank $r$ vector bundle $E,$ first choose a filtration of $E$ such that the successive quotients are line bundles and use the 
 fact that $\mathcal{F}_2(\oplus M_k) = \oplus \mathcal{F}_2(M_k)$ where $M_k$'s are line bundles over $C.$ Then the Chern character of $\mathcal{F}_2(E)$ has the follwing expression \cite{BL}: 
\[
 ch(\mathcal{F}_2(E)) = \text{degree}(E)(1- \text{exp}(-x)) -r(g-1) + r(1 +g + \theta)\text{exp}(-x).
\]
From the above expression one can easily see that 
 $ c_1(\mathcal{F}_2(E)) = (d -r(g+1))x + r \theta,$ where $d=$ degree $E.$
\end{rem}

\section{Semistability of $\mathcal{F}_2(E)$, for degree E even}

Let $C$ be a smooth irreducible projective curve over the field of complex numbers $\mathbb C$ of genus $g$ and let $E$ be a rank $r$ vector bundle on $C.$
In this section we will prove the semistability of the vector bundle $\mathcal{F}_2(E)$, when $r =2$ and degree $E$ is even. We start with the following definitions.
\begin{Def}
 Let $C$ be a non-singular irreducible curve. For a vector bundle $F$ on $C$ we define $$\mu(F) := \frac{\text{degree}(F)}{\text{rank}(F)}.$$
 A vector bundle $F$ on $C$ is said to be semistable (respectively, stable) if for every subbundle $F'$ of $F$ we have
 $$\mu(F') \leq \mu(F) (\text{respectively,} \mu(F') < \mu(F)).$$
\end{Def}

\begin{Def}
 Let $X$ be a smooth irreducible surface and let $H$ be an ample divisor on $X.$ For a coherent torsion free sheaf $F$ on $X,$ we set 
 $$\mu_H(F):= \frac{\text{degree}_H(F)}{\text{rank}(F)}$$ where $\text{degree}_H(F) = c_1(F)\cdot H.$
 
 A vector bundle $F$ on $X$ is said to be $\mu_H$-semistable (respectively, $\mu_H$-stable), if for every coherent torsion free subsheaf $F'$ of $F$ with 
 $0 < \text{rank}(F) < \text{rank} (E),$ we have 
 $$\mu_H(F') \leq \mu_H(F) (\text{respectively,} \mu_H(F') < \mu_H(F)).$$
\end{Def}

\begin{thm}\label{T}
 Let $E$ be a rank two stable vector bundle of even degree $d \geq 2$ on $C$ such that $\mathcal{F}_2(E)$ is globally generated. 
 Then the bundle $\mathcal{F}_2(E)$ on $S^2(C)$ is $\mu_{H'}$-semistable with respect to the ample class $H' = x + C$.
\end{thm}

\begin{rem}
 If $E$ is an even degree vector bundle which is a quotient of direct sum of very ample line bundles, i.e. if there is a surjection $\oplus L_i \to E$ where each $L_i$ is a 
 very ample line bundle on $C,$ then $E$ satisfies the property of Theorem $3.1$.
\end{rem}

We recall some well known results.

\begin{lem}
 Let $f : X \longrightarrow Y$ be a finite surjective morphism of non-singular surfaces, $F$ be a vector bundle on $Y$, and $H$ be an ample divisor on $Y$. Assume 
 $f^*(F)$ is $\mu_{f^*(H)}$-semistable 
 (repectively, $\mu_{f^*(H)}$-stable). Then $F$ is $\mu_H$-semistable (respectively, $\mu_H$-stable).
\end{lem}

{\bf Proof}: \cite[Lemma\, 4.4] {LMN}.

\begin{lem}
 Let $C$ be a smooth irreducible curve of genus $g \geq 1$ and let $K_C$ be the canonical bundle of $C$. Let $J^{g - 1}(C)$ be the variety of line bundles of degree $g - 1$ of $C$, and let 
 $\Theta$ be the divisor on $J^{g - 1}(C)$ consisting of line bundles with non-zero sections. Let $\xi$ be a line bundle on $C$ of degree $g - 3$ and 
 $$
 \nu_{\xi} : C \times C \longrightarrow J^{g - 1}(C)
 $$
 be the morphism $(x, y) \mapsto \mathcal{O}_{C \times C}(x + y) \otimes \xi$. Then
 $$
 \nu_{\xi}^*(\Theta) \cong p_1^*(K_C \otimes \xi^*) \otimes p_2^*(K_C \otimes \xi^*) \otimes \mathcal{O}_{C \times C}(- \Delta)
 $$
 where $\Delta$ is the diagonal of $C \times C$ and $p_i : C\times C \longrightarrow C$ is the $i$-th coordinate projection.
\end{lem}

{\bf Proof}: \cite[Lemma\, 4.5]{LMN}.

\bigskip

Using Lemma $3.5$, we see that, to prove the semistablity of $\mathcal{F}_2(E)$ on $S^2(C)$ with respect to the ample class $x + C$, 
it is sufficient to prove the semistability of 
$\pi^*(\mathcal{F}_2(E))$ on $C \times C$ with respect to the ample divisor $H:= \pi^*(H') = [x \times C + C \times x]$.
By Lemma $3.6$, we have $\pi^*(\theta) = 
(g + 1)[x \times C + C \times x] - \Delta$. Since $c_1(\mathcal{F}_2(E)) = (d - 2(g +1))x + 2\theta$, we have  
$$
c_1(\pi^*(\mathcal{F}_2(E))) = d[x \times C + C \times x] -  2\Delta,
$$
and
$$
\mu_{H}(\pi^*(\mathcal{F}_2(E))) = \frac{d - 2}{2}.
$$

\bigskip

First note that the bundle $\pi^*(\mathcal{F}_2(E))$ fits in the following exact sequence on $C \times C$:
\begin{equation}
 0 \rightarrow \pi^*(\mathcal{F}_2(E)) \rightarrow p_1^*(E) \oplus p_2^*(E) \xrightarrow{q} E = p_1^*(E)|_{\Delta} = p_2^*(E)|_{\Delta} \rightarrow 0
\end{equation}
where the map $q$ is given by $q : (u, v) \mapsto u|_{\Delta} - v|_{\Delta}$. Let $\phi_i : \pi^*(\mathcal{F}_2(E)) \rightarrow p_i^*(E)$ be the restriction of the 
projection $p_1^*(E) \oplus p_2^*(E) \longrightarrow p_i^*(E)$ to $ \pi^*(\mathcal{F}_2(E)) \subset p_1^*(E) \oplus p_2^*(E)$. Then from the exact sequence (1), we get the following two 
exact sequences:
\begin{equation}
 0 \rightarrow p_1^*(E) \otimes \mathcal{O}_{C \times C}(- \Delta) \rightarrow \pi^*(\mathcal{F}_2(E)) \xrightarrow{\phi_1} p_2^*(E) \rightarrow 0,
\end{equation}
and
\begin{equation}
 0 \rightarrow p_2^*(E) \otimes \mathcal{O}_{C \times C}(- \Delta) \rightarrow \pi^*(\mathcal{F}_2(E)) \xrightarrow{\phi_2} p_1^*(E) \rightarrow 0
\end{equation}
\cite[Section\, 3]{BN}.

\begin{lem}
  $p_i^*(E)$ is $\mu_{H}$-stable, $\forall i = 1, 2$.
\end{lem}
\begin{proof}
Due to symmetry, we will do it only for ${p_2}^*E$. Since over a smooth irreducible projective surface double dual of a torsion free sheaf is free, by taking double dual if necessary, we see 
that to prove stability or semistability it is enough to consider subsheafs which are line bundles.
Let $L$ be a line bundle on $C \times C$ which is a subsheaf of ${p_2}^*E$ such that the quotient, $M$ say,  is torsion free. We have an 
exact sequence
$$
0 \longrightarrow L \longrightarrow {p_2}^*E \longrightarrow M \longrightarrow 0.
$$
We restrict this exact sequence to $x \times C$ and $C \times x$, respectively, to obtain the following exact sequences
$$
0 \longrightarrow L|_{x \times C} \longrightarrow E \longrightarrow M|_{x \times C} \longrightarrow 0,
$$
and
$$
0 \longrightarrow L|_{C \times x} \longrightarrow \mathcal{O}_C \oplus \mathcal{O}_C \longrightarrow M|_{C \times x} \longrightarrow 0.
$$
From the first exact sequence we get, deg$(L|_{x\times C}) = c_1(L).[x \times C] < \mu(E) = \frac{d}{2}$, since $E$ is stable. And from the second exact sequence we get deg$(L|_{C \times x}) = 
c_1(L).[C \times x] \leq 0$. Thus deg$(L) =  c_1(L).[x \times C + C \times x] < \frac{d}{2}= \mu_H(p_2^*E)$, proving the Lemma.
\end{proof}
\bigskip

{\bf Proof of Theorem \ref{T}}:\\

Let $L$ be a line bundle which is a subsheaf of $\pi^*(\mathcal{F}_2(E)) $ such that the quotient is torsion free.
Suppose there is a non-zero homomorphism from $L$ to $p_1^*(E)(- \Delta) := 
p_1^*(E) \otimes \mathcal{O}_{C \times C}(- \Delta)$. Then 
$\mu_{H}(L) < \mu_{H}(p_1^*(E)(- \Delta)) = \frac{d - 4}{2} < \frac{d - 2}{2}$. 
So assume that there is no non-zero map from $L$ to $p_1^*(E)(- \Delta)$. Thus there is an injection $L \rightarrow p_2^*(E)$ 
so that $\mu_{H}(L) < \mu_{H}(p_2^*(E)) = 
\frac{d}{2}$. Since $d$ is even, $\mu_{H}(L) \leq \frac{d}{2} - 1 = \frac{d - 2}{2}$.

\bigskip

Now let $F$ be a rank two coherent subsheaf of $\pi^*(\mathcal{F}_2(E))$ such that quotient is torsion-free. 
Then we have the following commutative diagram:

\begin{center}

 \begin{tikzpicture}[node distance=2.5cm, auto]
  \node(U1){0};
  \node(U2) [right of = U1]{$p_1^*(E)(- \Delta)$};
  \node(U3) [right of = U2]{$\pi^*(\mathcal{F}_2(E))$};
  \node(U4) [right of = U3]{$p_2^*(E)$};
  \node(U5) [right of = U4]{$0$};
  \node(L1) [below=.6cm of  U1]{$0$};
  \node(L2) [below=.5cm of U2]{$F'$};
  \node(L3) [below=.5cm of  U3]{$F$};
  \node(L4) [below=.5cm of  U4]{$F''$};
  \node(L5) [below=.6cm of  U5]{$0$};
  \draw [->] (L2) to node {} (U2);
  \draw [->] (L3) to node {} (U3);
  \draw [->] (L4) to node {} (U4);
  \draw [->] (U1) to node {} (U2);
  \draw [->] (U2) to node {} (U3);
  \draw [->] (U3) to node {} (U4);
  \draw [->] (U4) to node {} (U5);
  \draw [->] (L1) to node {} (L2);
  \draw [->] (L2) to node {} (L3);
  \draw [->] (L3) to node {} (L4);
  \draw [->] (L4) to node {} (L5);
 \end{tikzpicture}

\end{center}
where the vertical arrows are injections. Suppose that both $F'$ and $F''$ are non-zero. These two are rank $1$ coherent sheaf.
So we have, deg$(F') = \mu_{H}(F') <  \mu_{H}(p_1^*(E)(- \Delta)) 
= \frac{d - 4}{2}$ and deg$(F'') = \mu_{H}(F'') < \mu_{H}(p_2^*(E)) = \frac{d}{2}$.
Thus $\mu_{H}(F) = \frac{1}{2}($deg$(F') +$ deg$(F'')) < \frac{d - 2}{2}$. 
Now assume at least one of $F'$ and $F''$ is zero. First let $F''$ be zero. Then we have an injection $F \rightarrow p_1^*(E)(- \Delta)$ and the 
cokernel is a torsion sheaf. If the cokernel is supported at only finitely many points, then 
$\mu_{H}(F) = \mu_{H}(p_1^*(E)(- \Delta)) < \frac{d - 2}{2}$. If the 
cokernel is supported at a co-dimension 1 subscheme, then $\mu_{H}(F) < \mu_{H}(p_1^*(E)(- \Delta)) < \frac{d - 2}{2}$.
Now let $F'$ is zero. So we have an 
injection $F \rightarrow p_2^*(E)$ and the cokernel is a torsion sheaf. If the cokernel 
is supported at a subscheme of co-dimension 1, then $\mu_{H}(F) < \mu_{H}(p_2^*(E)) = \frac{d}{2}$ so that 
$\mu_{H}(F) \leq \frac{d - 1}{2}$. If $\mu_{H}(F) = \frac{d-1}{2}$, then the cokernel is supported on a divisor
of degree one. Now  an effective  divisor of degree one  on $C \times C$ is of the form $x \times C$ or $C \times x$, for some
$x \in C$. Thus $c_1(F)$ is of the form $c_1(p_2^*(E)) + [-x \times C]$ or $c_1(p_2^*(E)) + [-C \times x]$. 
But $c_1(\pi^*(\mathcal{F}_2(E)) = d[C \times x + x \times C] - 2 \Delta$, therefore 
$c_1((\pi^*(\mathcal{F}_2(E)/F)) = (d+1)[ x \times C] -2\Delta \text{ or } d[x \times C] + [C \times x] - 2\Delta$. In both the 
cases the torsion free sheaf $\pi^*(\mathcal{F}_2(E)/F$ restricted to any curve of the form $x \times C$ has negative degree.
This gives a contradiction to the fact that $\pi^*(\mathcal{F}_2(E)$ is generated by sections. Thus we have,
$\mu_{H}(F) \le \frac{d-2}{2}$.

If the cokernel is 
supported only at finitely many points then $\mu_{H}(F) = \mu_{H}(p_2^*(E)) = \frac{d}{2}$. In this case, $F$ is a rank two stable sheaf and hence it is isomorphic to 
$p_2^*(E)$. So the exact sequence $(2)$ splits, i.e., $\pi^*(\mathcal{F}_2(E)) \cong p_1^*(E)(- \Delta) \oplus p_2^*(E)$. Since $p_1^*(E)|_{x \times C}$ is trivial,
deg$(p_1^*(E)(- \Delta)|_{x \times C}) < 0$. This contradicts the fact 
that $\mathcal{F}_2(E)$ and hence $\pi^*(\mathcal{F}_2(E))$ is globally generated.

\bigskip

Let $F$ be a rank $3$ coherent subsheaf of $\pi^*(\mathcal{F}_2(E))$ such that the quotient is torsion free. Then we have the following commutative diagram:

\begin{center}

\begin{tikzpicture}[node distance=2.5cm, auto]
  \node(U1){0};
  \node(U2) [right of = U1]{$p_1^*(E)(- \Delta)$};
  \node(U3) [right of = U2]{$\pi^*(\mathcal{F}_2(E))$};
  \node(U4) [right of = U3]{$p_2^*(E)$};
  \node(U5) [right of = U4]{0};
  \node(L1) [below=.6cm of  U1]{0};
  \node(L2) [below=.5cm of U2]{$F'$};
  \node(L3) [below=.5cm of  U3]{F};
  \node(L4) [below=.5cm of  U4]{F''};
  \node(L5) [below=.6cm of  U5]{0};
  \draw [->] (L2) to node {} (U2);
  \draw [->] (L3) to node {} (U3);
  \draw [->] (L4) to node {} (U4);
  \draw [->] (U1) to node {} (U2);
  \draw [->] (U2) to node {} (U3);
  \draw [->] (U3) to node {} (U4);
  \draw [->] (U4) to node {} (U5);
  \draw [->] (L1) to node {} (L2);
  \draw [->] (L2) to node {} (L3);
  \draw [->] (L3) to node {} (L4);
  \draw [->] (L4) to node {} (L5);
 \end{tikzpicture}

\end{center}

 where the vertical arrows are injections. We have two possibilities: (I) rank$F' = 2$ and rank$F'' = 1$; (II) rank$F' = 1$ and rank$F'' = 2$. 
Suppose that rank$F' = 2$ and rank$F'' = 1$. By the arguments above, we have, $\mu_{H}(F') \leq \frac{d - 4}{2}$ and 
$\mu_{H}(F'') < \frac{d}{2}$. So  
$$
\mu_{H}(F) < \frac{3d - 8}{6} < \frac{d - 2}{2}.
$$
Now assume that rank$F' = 1$ and rank$F'' = 2$. In this case, we have, $\mu_{H}(F') < \frac{d - 4}{2}$ and
$\mu_{H}(F'') \leq \frac{d}{2}$. If $d$ is even, 
$\mu_{H}(F') \leq \frac{d - 4}{2} - 1$, hence $\mu_{H}(F) \leq \frac{3d - 6}{6}= \frac{d - 2}{2}$.

\section{semistability of $\pi^*(\mathcal{F}_2(E))$ for degree E odd}\label{s3}

In this section we will prove that the semi-stability of $\pi^*(\mathcal{F}_2(E))$ when degree $E$ is odd. First let's recall some 
definitions. 

\begin{Def} Let $E$ be a non-zero vector bundle on $C$ and $k \in \mathbb{Z}$, we denote by $\mu_k(E)$ the 
rational number

$$\mu_k(E) := \frac{\text{degree}(E) + k}{\text{rank}(E)}.$$ 
We say that the vector bundle $E$ is $(k, l)$-stable (resp. $(k, l)$-semistable) if, for every proper subbundle $F$ 
 of $E$ we have
 \[
  \mu_k(F) < \mu_{-l}(E/F) (resp. \mu_k(F) \le \mu_{-l}(E/F)).
 \]
\end{Def}

 Note that usual Mumford stability is equivalent to $(0, 0)$-stability.
If $g \ge 3$, then there always exists a $(0, 1)$-stable bundle and if $g \ge 4$, then the set of $(0, 1)$-stable
bundles form a dense open subset of the moduli space of stable bundels over $C$ of rank $2$ and degree $d$. \cite[Section 5]{RN}

\begin{thm}
 Assume the genus of $C$ greater than 2. Let $E$ be a rank two $(0, 1)$-stable bundle of odd
 degree $d \geq 1$ on $C$ such that $\mathcal{F}_2(E)$ is globally generated.
  Then the bundle $\mathcal{F}_2(E)$ on $S^2(C)$ is $\mu_{H'}$-semistable with respect to the ample class $H'=x+C$.
\end{thm}
\begin{proof}
 Let $L$ be a line bundle which is a subsheaf of $\pi^*(\mathcal{F}_2(E)) $ such that the quotient is torsion free.
Suppose there is a non-zero homomorphism from $L$ to $p_1^*(E)(- \Delta) := 
p_1^*(E) \otimes \mathcal{O}_{C \times C}(- \Delta)$. Then 
$\mu_{H}(L) < \mu_{H}(p_1^*(E)(- \Delta)) = \frac{d - 4}{2} < \frac{d - 2}{2}$. 
So assume that there is no non-zero map from $L$ to $p_1^*(E)(- \Delta)$. Thus there is an injection $L \rightarrow p_2^*(E).$
 Now consider the exact sequence,
 \begin{equation}\label{eqA}
  0 \longrightarrow L \longrightarrow p_2^*(E) \longrightarrow M \longrightarrow 0,
 \end{equation}
where $M$ is a sheaf of rank $1$. 
Restricting the above exact sequence to $C \times x$, we see that, $c_1(L).[C \times x] \le 0$. On the other hand, restricting the 
above exact sequence to $C \times x$ and using that $E$ is (0.1)-stable, we get that $c_1(L).[C\times x] < \frac{d-1}{2}.$ Since 
$L$ is a line bundle, $c_1(L).[C\times x] \leq \frac{d-3}{2}.$ So we have $\mu_H(L) \leq \frac{d-3}{2}<\frac{d-2}{2}.$

\bigskip

Let's assume $F$ be a rank two coherent subsheaf of $\pi^*(\mathcal{F}_2(E))$ such that quotient is torsion-free. 
Then we have the following commutative diagram:

\begin{center}

 \begin{tikzpicture}[node distance=2.5cm, auto]
  \node(U1){0};
  \node(U2) [right of = U1]{$p_1^*(E)(- \Delta)$};
  \node(U3) [right of = U2]{$\pi^*(\mathcal{F}_2(E))$};
  \node(U4) [right of = U3]{$p_2^*(E)$};
  \node(U5) [right of = U4]{$0$};
  \node(L1) [below=.6cm of  U1]{$0$};
  \node(L2) [below=.5cm of U2]{$F'$};
  \node(L3) [below=.5cm of  U3]{$F$};
  \node(L4) [below=.5cm of  U4]{$F''$};
  \node(L5) [below=.6cm of  U5]{$0$};
  \draw [->] (L2) to node {} (U2);
  \draw [->] (L3) to node {} (U3);
  \draw [->] (L4) to node {} (U4);
  \draw [->] (U1) to node {} (U2);
  \draw [->] (U2) to node {} (U3);
  \draw [->] (U3) to node {} (U4);
  \draw [->] (U4) to node {} (U5);
  \draw [->] (L1) to node {} (L2);
  \draw [->] (L2) to node {} (L3);
  \draw [->] (L3) to node {} (L4);
  \draw [->] (L4) to node {} (L5);
 \end{tikzpicture}

\end{center}
where the vertical arrows are injections. We need to consider three different cases: (I) rank $F' = 1=$ rank $F''$; 
(II) $F''=0$; (III) $F'=0$. In each of these cases, we can argue exaclty as in the case of even degree to conclude that 
$\mu_H(F) \leq \frac{d-2}{2} =\mu_H(\pi^*\mathcal{F}_2(E)).$

\bigskip

Now assume $F$ is subsheaf of $\pi^*\mathcal{F}_2(E)$ rank $3.$  Then again we have the following commutative diagram:

\begin{center}

\begin{tikzpicture}[node distance=2.5cm, auto]
  \node(U1){0};
  \node(U2) [right of = U1]{$p_1^*(E)(- \Delta)$};
  \node(U3) [right of = U2]{$\pi^*(\mathcal{F}_2(E))$};
  \node(U4) [right of = U3]{$p_2^*(E)$};
  \node(U5) [right of = U4]{0};
  \node(L1) [below=.6cm of  U1]{0};
  \node(L2) [below=.5cm of U2]{$F'$};
  \node(L3) [below=.5cm of  U3]{F};
  \node(L4) [below=.5cm of  U4]{F''};
  \node(L5) [below=.6cm of  U5]{0};
  \draw [->] (L2) to node {} (U2);
  \draw [->] (L3) to node {} (U3);
  \draw [->] (L4) to node {} (U4);
  \draw [->] (U1) to node {} (U2);
  \draw [->] (U2) to node {} (U3);
  \draw [->] (U3) to node {} (U4);
  \draw [->] (U4) to node {} (U5);
  \draw [->] (L1) to node {} (L2);
  \draw [->] (L2) to node {} (L3);
  \draw [->] (L3) to node {} (L4);
  \draw [->] (L4) to node {} (L5);
 \end{tikzpicture}

\end{center}

where the vertical arrows are injections. We have two possibilities: (I) rank$F' = 2$ and rank$F'' = 1$; (II) rank$F' = 1$ 
and rank$F'' = 2$. Using the same argument as in Theorem $3.3$, we can show that in the case of (I), $\mu_H(F) < \frac{d-2}{2}.$ 
Now consider the case (II). In this case, restircting the exact sequence $0 \to F' \to p_1^*(E)(- \Delta)$ to 
$x \times C$ and $C \times x,$ we get that
$$
0 \to F'|_{x\times C} \to \mathcal{O}_C(-x)
$$
and
$$
0 \to F'|_{C \times x} \to E \otimes \mathcal{O}_C(-x).
$$
From these two exact sequences and using the fact that $E$ is $(0,1)$-stable we see that $\mu_H(F') < \frac{d-4}{2}$ and 
hence $\mu_H(F') \leq \frac{d-6}{2}.$ Also using the same argument as above, we have, in any case, $\mu_H(F'') \leq \frac{d}{2}.$
Combining all these, we get that $\mu_H(F) < \frac{d-2}{2}.$

\end{proof}

\section{Restriction to curves of the form $x+C$}

In this section we will investigate the restriction of $\mathcal{F}_2(E)$ to the curves of the form $x+C$ where $x+C$ is the 
reduced divisor of $S^2(C)$ whose support equals to $\{x+c : c \in C\}.$ For this we have the follwoing theorem.

\begin{thm}
 Let $C$ be a smooth irreducible projective curve over $\mathbb{C}$ of genus $g$ and let $E$ be a rank to vector bundle on $C$ of 
 degree $d \geq 3.$  Then for any $x \in C, \mathcal{F}_2(E)|_{x+C}$ is not semistable.
\end{thm}

\begin{proof}
 First note that, since $E$ is locally free, $p_2^*E$ is flat over $S^2(C)$ and using the base change formula we get 
 $$\mathcal{F}_2(E)|_{x+C}=\pi_*(p_2^*E|_{\pi^{-1}(x+C)}).$$
 Also we have the following exact sequence
 $$0 \to p_2^*E|_{\pi^{-1}(x+C)} \to p_2^*E|_{x \times C} \oplus p_2^*E|_{C \times x} \to E|_{(x,x)} \to 0.$$
 From this exact sequence and using the fact that $\pi|_{x \times C} : x\times C \to x+C$ and $\pi|_{C \times x} : C \times x 
 \to x+C$ are isomorphisms and $p_2^*E|_{x \times C} = E$ and $p_2^*E|_{C \times x} = E_x \otimes \mathcal{O}_C$, we get an 
 injective map $$0 \to E \otimes \mathcal{O}_C(-x) \to \mathcal{F}_2(E)|_{x+C}.$$
 Now the degree of $E \otimes \mathcal{O}_C(-x) = d-2$ and that of $\mathcal{F}_2(E)|_{x+C} =d-2.$ So the cokernel is rank $2$ 
 coherent sheaf of degree zero. If it is torsion free then clearly $\mathcal{F}_2(E)|_{x+C}$ is not semistable. If the 
 cokernel has torsion, then there is an effective divisor $D$ such that the above map factors through 
 $E \otimes \mathcal{O}_C(-x) \otimes \mathcal{O}_C(D)$ and in this case the cokernel will be again torsion free. But in this case 
 the degree of the cokernel will be of negetive degree. So in this case $\mathcal{F}_2(E)|_{x+C}$ has a torsion free quotient of 
 negetive degree. Hence it is not semistable.
\end{proof}

$Acknowledgements:$
 We would like to thank Prof. D.S. Nagaraj for encouraging us to choose the problem and also letting us the helping hand 
 whenever necessary. The first named author would like to thank ISI Bangalore for their hospitality during the stay where 
 the work started.


\begin{thebibliography}{1111}

\bibitem{ACGH} E. Arbarello, M. Cornalba, P. Griffiths and J. Harris: Geometry of Algebraic curves I. Grundl. der Math. W. 267, Berlin-Heidelberg-New York 1985.


\bibitem{LMN} El Mazouni, A.; Laytimi, F.; Nagaraj, D. S. $\textit{Secant bundles on second symmetric power of a curve}$. J. Ramanujan Math. Soc. 26 (2011), no. 2, 181-194.

\bibitem{BN2}  Biswas, Indranil; Nagaraj, D. S. $\textit{Stability of secant bundles on second}$ $\textit{symmetric power of a curve}$. Commutative algebra and algebraic geometry (CAAG-2010), 
13-18, Ramanujan Math. Soc. Lect. Notes Ser., 17, Ramanujan Math. Soc., Mysore, 2013.

\bibitem{BN}  Biswas, Indranil; Nagaraj, D. S. $\textit{Reconstructing vector bundles on}$ $\textit{curves from their direct image on symmetric powers}$. Arch. Math. (Basel) 99 (2012), no. 4, 327-331.

\bibitem{BL} Biswas, Indranil; Laytimi, Fatima $\textit{Direct image and parabolic structure on}$ $\textit{symmetric product of curves}$. J. Geom. Phys. 61 (2011), no. 4, 773-780.

\bibitem{RN} Narasimhan, M. S.; Ramanan, S. $\textit{Geometry of Hecke cycles. I}$. C. P. Ramanujam$-$a tribute, pp. 291-345, Tata Inst. Fund. Res. Studies in Math., 8, Springer, Berlin-New York, 1978.
\end{thebibliography}
\end{document}